\documentclass[a4paper,12pt,final]{amsart}
\usepackage{times,a4wide,mathrsfs,amssymb,amsmath,amsthm,enumerate,xypic,tikzsymbols,dsfont,textcomp}

\newcommand{\C}{\mathbb{C}}
\newcommand{\ZZ}{\mathbb{Z}}

\newcommand{\QQ}{\mathbb{Q}}
\newcommand{\NN}{\mathbb{N}}
\newcommand{\PP}{\mathbb{P}}

\newcommand{\OO}{\mathcal O}

\newcommand{\XX}{\mathcal X}
\newcommand{\YY}{\mathcal Y}

\newcommand{\VV}{\mathcal V}
\newcommand{\WW}{\mathcal W}

\newcommand{\MM}{\mathcal M}

\newcommand{\wt}{\widetilde}

\newcommand{\one}{\mathds{1}}

\DeclareMathOperator{\Pic}{Pic}
\DeclareMathOperator{\ima}{Im}

\DeclareMathOperator{\OGr}{OGr}

\DeclareMathOperator{\spin}{Spin}

\DeclareMathOperator{\CH}{CH}
\DeclareMathOperator{\A}{A}
\DeclareMathOperator{\GDA}{GDA}
\DeclareMathOperator{\B}{B}

\newtheorem{theorem}{Theorem}[section]

\newtheorem{lemma}[theorem]{Lemma}

\newtheorem{corollary}[theorem]{Corollary}
\newtheorem{proposition}[theorem]{Proposition}

\newtheorem{remark}[theorem]{Remark}
\newtheorem{definition}[theorem]{Definition}
\newtheorem{convention}{Conventions}
\newtheorem{question}[theorem]{Question}
\newtheorem{notation}[theorem]{Notation}

\newtheorem{nonumbering}{Theorem}

\newtheorem{nonumberingp}{Proposition}

\newtheorem{nonumberingt}{Acknowledgements}

\begin{document}

\author[Robert Laterveer]
{Robert Laterveer}

\address{Institut de Recherche Math\'ematique Avanc\'ee,
CNRS -- Universit\'e 
de Strasbourg,\
7 Rue Ren\'e Des\-car\-tes, 67084 Strasbourg CEDEX,
FRANCE.}
\email{robert.laterveer@math.unistra.fr}

\title[]{Algebraic cycles and Fano threefolds of genus 7}

\begin{abstract} Let $Y$ be a very general prime Fano threefold of genus 7. We exhibit an explicit 2-cycle on $Y\times Y$ that is Abel--Jacobi trivial but non-torsion in the Chow group $\A^4(Y\times Y)$.
As a consequence, $Y$ does not admit a multiplicative Chow--K\"unneth decomposition, in the sense of Shen--Vial. 

We also show that any Fano threefold has a multiplicative Chow--K\"unneth decomposition modulo algebraic equivalence.
 \end{abstract}

\thanks{Supported by ANR grant ANR-20-CE40-0023.}

\maketitle

\section{Introduction}

Given a smooth projective variety $Y$ over $\C$, let $\A^i(Y):=\CH^i(Y)_{\QQ}$ denote the Chow groups of $Y$ (i.e. the groups of codimension $i$ algebraic cycles on $Y$ with $\QQ$-coefficients, modulo rational equivalence). The intersection product defines a ring structure on $\A^\ast(Y)=\bigoplus_i \A^i(Y)$, the {\em Chow ring\/} of $Y$ \cite{F}. 

In the special case of K3 surfaces, this ring structure has remarkable properties:

\begin{theorem}[Beauville--Voisin \cite{BV}]\label{bv} Let $S$ be a projective K3 surface. 
The $\QQ$-subalgebra
  \[   \bigl\langle  \A^1(S), c_j(S) \bigr\rangle\ \ \ \subset\ \A^\ast(S) \]
  injects into cohomology under the cycle class map.
  \end{theorem}

\begin{theorem}[Voisin \cite{V17}, Yin \cite{Yin}]\label{vy} Let $S$ be a projective K3 surface, and $m\in\NN$. The $\QQ$-subalgebra
  \[   R^\ast(S^m):=\bigl\langle \A^1(S), \Delta_S\bigr\rangle\ \ \ \subset\ \A^\ast(S^m) \]
  (generated by pullbacks of divisors and pullbacks of the diagonal $\Delta_S\subset S\times S$)
  injects into cohomology under the cycle class map for all $m\le 2\dim H^2_{tr}(S,\QQ)+1$ (where $H^2_{tr}(S,\QQ)$ denotes the transcendental part of cohomology). Moreover, $R^\ast(S^m)$ injects into cohomology for all $m\in\NN$ if and only if $S$ is Kimura finite-dimensional.
  \end{theorem}

The Chow ring of abelian varieties also has an interesting property: there is a multiplicative splitting, defined by the Fourier transform \cite{Beau}.

Motivated by this particular behaviour of K3 surfaces and abelian varieties, Beauville \cite{Beau3} has conjectured that for certain special varieties, the Chow ring should admit a multiplicative splitting. In the wake of this ``splitting property conjecture'' of Beauville's, Shen--Vial \cite{SV} have introduced the concept of {\em multiplicative Chow--K\"unneth decomposition\/}
(we will abbreviate this to ``MCK decomposition''). With the concept of MCK decomposition, it is possible to make concrete sense of the elusive ``splitting property conjecture'' of Beauville.

It seems difficult to understand exactly which varieties admit an MCK decomposition. To give an idea of what is known: hyperelliptic curves have an MCK decomposition, but the very general curve of genus $\ge 3$ does not have an MCK decomposition (cf. Subsection \ref{mckc} below); K3 surfaces have an MCK decomposition, but certain high degree surfaces in $\PP^3$ do not have an MCK decomposition (cf. the examples given in \cite{OG}). 

The present work is part of a program that aims to understand MCK decompositions for Fano threefolds.
In earlier work I had raised the following question:

\begin{question}[\cite{g8}]\label{ques} Let $Y$ be a Fano threefold with Picard number 1. Does $Y$ admit an MCK decomposition?
\end{question} 

This question was prompted by examples of Beauville \cite[Examples 2.1.5(a)]{Beau3} of Fano threefolds $Y$ for which
$\A^1(Y)\cdot \A^1(Y)\subset \A^2(Y)$ does not inject into cohomology. This property implies that $Y$ does not have an MCK decomposition \cite[Example 2.11]{FLV2}.
Beauville's obstruction disappears when $Y$ has Picard number 1, and so Question \ref{ques} is natural. (Also, this shows that if one expects a negative answer to Question \ref{ques} one needs to find a different obstruction to MCK.)

The answer to Question \ref{ques} is affirmative for cubic threefolds \cite{Diaz}, \cite{FLV2}, for intersections of 2 quadrics \cite{2q}, for intersections of a quadric and a cubic \cite{55}, for quartic and sextic double solids \cite{more}, as well as for prime Fano threefolds of genus 8 \cite{g8} and of genus 10 \cite{g10}.

The main result of this paper provides a negative answer to Question \ref{ques}, by exhibiting an obstruction lying in $\A^4(Y\times Y)$:

\begin{nonumbering}[=Theorem \ref{main}] Let $Y$ be a very general prime Fano threefold of genus 7, and let $H:=-K_Y\in\A^1(Y)$. Then the cycle
 \[ \begin{split} Z_Y:=  \Delta_Y\cdot (p_1)^\ast(H) -   {1\over 12}&\Bigl( (p_1)^\ast(H )\cdot(p_2)^\ast(H^3)\\
       &+(p_1)^\ast(H^2)\cdot (p_2)^\ast(H^2 ) + (p_1)^\ast(H^3 )\cdot (p_2)^\ast(H )\Bigr)     \ \ \in\ \A^4(Y\times Y)\\
       \end{split} \]
   is (Abel--Jacobi trivial but) non-zero.
   
   Consequently, $Y$ does not admit an MCK decomposition.
\end{nonumbering}

This might be considered a negative result. On the other hand, the fact that there is an explicit cycle which is Abel--Jacobi trivial but non-torsion in the Chow group seems interesting; indeed, while one knows for general reasons that many such cycles exist, there are but few explicit examples of cycles of this kind. A famous example of this kind is the Faber--Pandharipande cycle on the square of a very general curve of genus at least 4 \cite{GG} --- and actually, the proof of Theorem \ref{main} is based on the non-triviality of this Faber--Pandharipande cycle for the curve dual to $Y$.

We also show that a weak version of Question \ref{ques} has a positive answer:

\begin{nonumberingp}[=Proposition \ref{main2}] Any Fano threefold admits an MCK decomposition modulo algebraic equivalence.
\end{nonumberingp}

This is in marked contrast to the case of curves: a very general curve of genus $\ge 4$ does not have an MCK decomposition modulo algebraic equivalence (as explained in Subsection \ref{mckc} below). Consequently, the motivic relation between a genus 7 Fano threefold and its dual curve does not behave well with respect to the multiplicative structure; this is exemplified in Corollary \ref{taut}.

 \vskip0.6cm

\begin{convention} In this paper, the word {\sl variety\/} will refer to a reduced irreducible scheme of finite type over $\C$. A {\sl subvariety\/} is a (possibly reducible) reduced subscheme which is equidimensional. 

{\bf All Chow groups will be with rational coefficients}: we will denote by $\A_j(Y)$ the Chow group of $j$-dimensional cycles on $Y$ with $\QQ$-coefficients; for $Y$ smooth of dimension $n$ the notations $\A_j(Y)$ and $\A^{n-j}(Y)$ are used interchangeably. 
The notation $\A^j_{hom}(Y)$ will be used to indicate the subgroup of homologically trivial cycles.

The contravariant category of Chow motives (i.e., pure motives with respect to rational equivalence as in \cite{Sc}, \cite{MNP}) will be denoted 
$\MM_{\rm rat}$.
\end{convention}

\section{Preliminaries}

\subsection{Prime Fano threefolds of genus 7} The classification of smooth Fano threefolds is one of the crowning glories of twentieth century algebraic geometry \cite{IP}. Fano threefolds that are {\em prime\/} (i.e. with Picard group of rank $1$ generated by the canonical divisor) come in 10 explicitly described families. In this paper we will be concerned with one of these families, that of genus 7 Fano threefolds.

\begin{notation} Given $V$ a 10-dimensional complex vector space equipped with a quadratic form, let $\spin(V)$ denote the algebraic group corresponding to the Dynkin diagram $D_5$. There are two dual
16-dimensional irreducible representations $S^+$ and $S^-$ (called the {\em half-spin representations}), corresponding to the roots $\alpha_+=\alpha_4$ resp. $\alpha_-=\alpha_5$.

Let
  \[ \Sigma^{\pm}:= \spin(V)/P(\alpha_{\pm}) \ ,\]
  where $P(\alpha_{\pm})$ denotes the parabolic subgroup associated to $\alpha_{\pm}$. The two symmetric spaces $\Sigma^+$ and $\Sigma^-$ can be identified with the two connected components of the orthogonal Grassmannian $\OGr(5,10)$, i.e.
    \[ \OGr(5,10)= \Sigma^+ \cup \Sigma^-\ .\]
 Following \cite{Ku3}, we will call $\Sigma^+$ the {\em spinor tenfold}. The spaces $\Sigma^+,\Sigma^-$ are projectively dual to one another (see for instance \cite{Ku3}, \cite{LM}).  
\end{notation}

\begin{theorem}[Mukai \cite{Mu1}, \cite{Mu2}]\label{muk} Let $Y$ be a prime Fano threefold (i.e., a smooth projective Fano threefold with $\Pic(Y)=\ZZ[K_Y]$) of genus $7$ (i.e. $-K_Y^3=12$). Then $Y$ is isomorphic to a dimensionally transverse intersection 
    \[ Y= \Sigma^+ \cap \PP(A)\ \ \subset\ \PP(S^+) \]
    for some 9-dimensional vector subspace $A\subset S^+$.

Conversely, any smooth dimensionally transverse intersection of the form $  \Sigma^+ \cap \PP(A) $ (where $A\subset S^+$ is a 9-dimensional vector space)
is a prime Fano threefold of genus $7$.

The Hodge diamond of $Y$ is
  \[ \begin{array}[c]{ccccccc}
      &&&1&&&\\
       &&0&&0&&\\
             &0&&1&&0&\\
                    0&&7&&7&&0\\
      &0&&1&&0&\\
      &&0&&0&&\\
      &&&1&&&\\
      \end{array}\]
  \end{theorem}

\begin{proof} The ``conversely'' statement is just because the spinor tenfold $\Sigma^+ $ is a Fano variety with Picard number 1, index $8$ and degree $12$; the codimension $7$ linear section $Y\subset \Sigma^+$ thus has index $1$ and degree $d=12$ (i.e. genus $g=d/2+1=7$).
The first statement is proven in \cite{Mu1} and \cite{Mu2}.

To see that $h^{2,1}(Y)=7$, one can look it up in \cite{IP},  or use Theorem \ref{dual} below. 
\end{proof}

To any Fano threefold of genus 7 one can associate a curve of genus 7:

\begin{definition} Let $Y= \Sigma^+ \cap \PP(A) $ be a prime Fano threefold of genus 7, where  $A\subset S^+$ is a 9-dimensional vector space.
Using the duality $(S^+)^\ast=S^-$, let us define
  \[   C:= \Sigma^-\cap \PP(A^\bot)\ \ \subset\ \PP(S^-)\ .\]
  Then $C$ is a smooth curve of genus 7, called the {\em dual\/} of $Y$.
\end{definition}

The following will be important below:

\begin{proposition}\label{surj} A general curve of genus 7 is the dual of a prime Fano threefold of genus 7.
\end{proposition}

\begin{proof} More precisely, it is proven by Mukai \cite{Mu1} that any curve of genus 7 having no linear system of type $g^1_4$ is obtained in this way. Since Brill--Noether loci are closed subsets, this defines an open subset of the moduli space of genus 7 curves.
\end{proof}

The Fano threefold $Y$ and the curve $C$ are related in multiple ways (we will not need all of this below, but we include it for completeness' sake):

\begin{theorem}\label{dual} Let $Y$ be a prime Fano threefold of genus 7, and $C$ its dual curve.

\begin{enumerate} 

\item The Jacobian of $C$ is isomorphic to the intermediate Jacobian of $Y$.

\item $Y$ and $C$ are related via homological projective duality, i.e. there is a semi-orthogonal decomposition of the derived category
  \[ D^b(Y)=\langle \OO_Y, U_Y, D^b(C) \rangle\ ,\]
  where $U_Y$ is a certain rank 5 vector bundle on $Y$.

\item The surface of conics in $Y$ is isomorphic to $C^{(2)}$.

\item The Hilbert scheme of cubic curves in $Y$ is isomorphic to $C^{(3)}$.

\item For general $Y$, the curve $C$ is isomorphic to the moduli space $M_Y(2,1,5)$ of rank 2 stable sheaves on $Y$ with $c_1=1, c_2=5$.

\item The general threefold $Y$ is isomorphic to the Brill--Noether locus of rank 2 stable bundles ${\mathcal E}$ on $C$ with determinant $K_C$ and $h^0(C,{\mathcal E})\ge 5$.

\end{enumerate}

\end{theorem}

\begin{proof} 

\noindent
\begin{enumerate}

\item This is proven in \cite{IM}. It readily follows from diagram \eqref{diag} below.

\item This is proven by Kuznetsov in \cite[Theorem 5.4]{Ku0}.

\item This is first proven by Kuznetsov \cite[Theorem 6.3]{Ku0}; an alternative proof can be found in \cite[Proposition 2.2]{IM2}.

\item This is \cite[Theorem 4.5]{BF}.

\item This is contained in \cite[Theorem 0.1]{IM}.

\item This is proven by Mukai \cite{Mu5}, and reproven in \cite[Proposition 4.6]{IM}.

\end{enumerate}
\end{proof}

For our purposes, we also need to relate $Y$ and $C$ on the level of Chow motives, thus adding one more item to the relations of Theorem \ref{dual}:

\begin{proposition}\label{dualchow} Let $Y$ be a prime Fano threefold of genus 7, and $C$ its dual curve.
There is an isomorphism of Chow motives
    \[ h(Y)\cong  h(C)(-1)\oplus \one\oplus \one(-3)\ \ \hbox{in}\ \MM_{\rm rat}\ .\]
\end{proposition}

\begin{proof} There are several ways one can go about this.

One possibility is to use the homological relation of Theorem \ref{dual}(1), check that it is induced by a correspondence and hence gives an isomorphism of homological motives
   \[ h(Y)\cong  h(C)(-1)\oplus \one\oplus \one(-3)\ \ \hbox{in}\ \MM_{\rm hom}\ .\]  
   Since both sides are Kimura-finite (i.e. finite-dimensional in the sense of \cite{Kim}; the Kimura-finiteness of $Y$ is a special case of \cite[Theorem 4]{43}), one may then upgrade to an isomorphism of Chow motives.
   
  Another possibility is to use the derived category result of Theorem \ref{dual}(2), which gives a correspondence-induced isomorphism of Chow groups
    \[ \A^2_{hom}(Y)\cong \A^1_{hom}(C)\ .\]
    Then the Bloch--Srinivas decomposition of the diagonal argument \cite{BS} yields the required isomorphism of motives.
    
 Alternatively, a more direct argument is as follows. As discovered by Iskovskikh--Prokhorov \cite[Theorem 4.4.11(iii)]{IP}, and exploited by Iliev--Markushevich \cite{IM}, \cite[Section 1, Diagram 2]{IM2}, there exist nice geometric ways of relating $Y$ and $C$. That is, given a prime genus 7 Fano threefold $Y$ with dual curve $C$, there is an explicit rationality construction for $Y$ given by the diagram
    \begin{equation}\label{diag} \begin{array}[c]{ccc}   \wt{Y} & { \buildrel  \phi \over {\dashrightarrow}} & \wt{Q}\\
                           &&\\
                         \ \ \   \downarrow{\sigma}&&\ \ \ \  \downarrow{\tau}\\
                           &&\\
                        q \subset   Y \ \ \ \ \ \ \ \ && \ \ \ \ \ \  \ \ \ \ \ \ \ Q \supset C_Q\ .\\
                        \end{array} \end{equation}
   In this diagram the morphism $\sigma$ is the blow-up of $Y$ with center a smooth conic $q\subset Y$, the morphism $\tau$ is the blow-up of the smooth 3-dimensional quadric $Q$ with center a genus 7 curve $C_Q\subset Q$ that is isomorphic to the dual curve $C$. The birational map $\phi$ is a flop where the flopping curves in $\wt{Y}$ are the (strict transforms of the) 14 lines meeting $q$ and the flopping curves in $\wt{Q}$ are the (strict transforms of the) 14 trisecants of the curve $C_Q$.
     
 One readily finds Chow--K\"unneth decompositions of $Y$ and $C$ with the property that
   \[ \begin{split} h(Y)&= \one\oplus \one(-1)\oplus h^3(Y) \oplus \one(-2) \oplus \one(-3)\ ,\\    
                       h(C_Q)&= \one\oplus h^1(C_Q)\oplus \one(-1)\ \ \ \ \ \ \ \ \ \ \ \hbox{in}\ \MM_{\rm rat}\ .\\
                       \end{split} \]
            Applying the blow-up formula (combined with the fact that $q$ and $Q$ have trivial motives), one obtains
            \[             \begin{split} h(\wt{Y})&=  h^3(Y) \oplus \bigoplus \one(\ast)\ ,\\    
                       h(\wt{Q})&=  h^1(C_Q)(-1)\oplus \bigoplus \one(\ast)\ \ \ \ \ \  \ \ \ \hbox{in}\ \MM_{\rm rat}\ .\\
                       \end{split} \]     
                 Let $Z\to\wt{Y}$ denote the blow-up of the 14 flopping lines, so that $Z$ dominates $\wt{Q}$. As lines have trivial motives, one obtains
                 \[ h^3(Y)\oplus   \bigoplus \one(\ast)= h(Z) = h^1(C_Q)(-1)\oplus \bigoplus\one(\ast)     \ \ \hbox{in}\ \MM_{\rm rat}\ . \]
   Here the motives   $   h^3(Y)$ and $h^1(C_Q)$ are oddly finite-dimensional, whereas trivial motives $\one(\ast)$ are evenly finite-dimensional (in the sense of Kimura \cite{Kim}).
   But the decomposition of a finite-dimensional motive into even and odd part is unique up to isomorphism \cite[Proposition 6.3]{Kim}, and so we obtain an isomorphism
     \[  h^3(Y) = h^1(C_Q)(-1)\ \ \hbox{in}\ \MM_{\rm rat}\ . \]
     Adding some trivial motives, this proves the proposition.
     
  (Remark: Iliev--Markushevich construct 2 other explicit birational maps involving $Y$ and $C$ \cite[Diagrams 1 and 3]{IM2}. One could just as well use one of those other birationalities in the proof above.)   
        \end{proof}

\subsection{MCK decomposition}
\label{ss:mck}

\begin{definition}[Murre \cite{Mur}] Let $X$ be a smooth projective variety of dimension $n$. We say that $X$ has a {\em CK decomposition\/} if there exists a decomposition of the diagonal
   \[ \Delta_X= \pi^0_X+ \pi^1_X+\cdots +\pi_X^{2n}\ \ \ \hbox{in}\ \A^n(X\times X)\ ,\]
  such that the $\pi^i_X$ are mutually orthogonal idempotents and $(\pi_X^i)_\ast H^\ast(X,\QQ)= H^i(X,\QQ)$.
  
  (NB: ``CK decomposition'' is shorthand for ``Chow--K\"unneth decomposition''.)
\end{definition}

\begin{remark} Murre has conjectured that any smooth projective variety should have a CK decomposition \cite{Mur}, \cite{J4}. 
\end{remark}

\begin{definition}[Shen--Vial \cite{SV}] Let $X$ be a smooth projective variety of dimension $n$, and let $\Delta_X^{sm}\in \A^{2n}(X\times X\times X)$ denote the class of the small diagonal
  \[ \Delta_X^{sm}:=\bigl\{ (x,x,x)\ \vert\ x\in X\bigr\}\ \subset\ X\times X\times X\ .\]
  An {\em MCK decomposition\/} is defined as a CK decomposition $\{\pi_X^i\}$ of $X$ that is {\em multiplicative\/}, i.e. it satisfies
  \[ \pi_X^k\circ \Delta_X^{sm}\circ (\pi_X^i\times \pi_X^j)=0\ \ \ \hbox{in}\ \A^{2n}(X\times X\times X)\ \ \ \hbox{for\ all\ }i+j\not=k\ .\]
  
 (NB: ``MCK decomposition'' is shorthand for ``multiplicative Chow--K\"unneth decomposition''.) 
  \end{definition}
  
  \begin{remark} The small diagonal (when considered as a correspondence from $X\times X$ to $X$) induces the {\em multiplication morphism\/}
    \[ \Delta_X^{sm}\colon\ \  h(X)\otimes h(X)\ \to\ h(X)\ \ \ \hbox{in}\ \MM_{\rm rat}\ .\]
 Let us assume $X$ has a CK decomposition
  \[ h(X)=\bigoplus_{i=0}^{2n} h^i(X)\ \ \ \hbox{in}\ \MM_{\rm rat}\ .\]
  By definition, this decomposition is multiplicative if for any $i,j$ the composition
  \[ h^i(X)\otimes h^j(X)\ \to\ h(X)\otimes h(X)\ \xrightarrow{\Delta_X^{sm}}\ h(X)\ \ \ \hbox{in}\ \MM_{\rm rat}\]
  factors through $h^{i+j}(X)$.
  
  If $X$ has an MCK decomposition, then setting
    \[ \A^i_{(j)}(X):= (\pi_X^{2i-j})_\ast \A^i(X) \ ,\]
    one obtains a bigraded ring structure on the Chow ring: that is, the intersection product sends $\A^i_{(j)}(X)\otimes \A^{i^\prime}_{(j^\prime)}(X) $ to  $\A^{i+i^\prime}_{(j+j^\prime)}(X)$.
    
      It is conjectured that for any $X$ with an MCK decomposition, one has
    \[ \A^i_{(j)}(X)\stackrel{??}{=}0\ \ \ \hbox{for}\ j<0\ ,\ \ \ \A^i_{(0)}(X)\cap \A^i_{hom}(X)\stackrel{??}{=}0\ ;\]
    this is related to Murre's conjectures B and D, that have been formulated for any CK decomposition \cite{Mur}, \cite{J4}.


For more background on the concept of MCK, and for examples of varieties with an MCK decomposition, we refer to \cite[Section 8]{SV}, as well as \cite{V6}, \cite{SV2}, \cite{FTV}, \cite{37},
  \cite{38}, \cite{39}, \cite{40}, \cite{44}, \cite{FLV2}, \cite{60}, \cite{55}, \cite{59}, \cite{NOY}.
    \end{remark}

\subsection{MCK and curves}
\label{mckc} 
In this subsection, we aim to show that the notion of MCK is interesting already for curves.

\begin{proposition}[Shen--Vial \cite{SV}]\label{hyper} Any hyperelliptic curve admits an MCK decomposition.
\end{proposition}

\begin{proof} This is \cite[Example 8.16]{SV}. The idea is that for a curve $C$, the CK decomposition defined by the choice of a degree 1 zero-cycle $z\in \A^1(C)$ is MCK if and only if the {\em modified diagonal\/} 
  \[\begin{split} \Gamma_{3}(C,
			z):=&\Delta^{sm}_{C}-p_{12}^{*}(\Delta_{C})p_{3}^{*}(z)-p_{23}^{*}(\Delta_{C})p_{1}^{*}(z)-p_{13}^{*}(\Delta_{C})p_{2}^{*}(z)\\
			&+p_{1}^{*}(z)p_{2}^{*}(z)+p_{1}^{*}(z)p_{3}^{*}(z)+p_{2}^{*}(z)p_{3}^{*}(z) \ \ \ \ \ \hbox{in}\ \A^2(C^3)\ \\
			\end{split}\]  
  vanishes. But Gross--Schoen \cite{GS} have shown the vanishing
     \[  \Gamma_3(C,z)=0\ \ \hbox{in} \ \A^2(C\times C\times C) \]
     for $C$ hyperelliptic and $z\in C$ a Weierstra{\ss}  point.
\end{proof}

  \begin{proposition}\label{no} A very general curve of genus $g\ge 3$ does not admit an MCK decomposition.
  \end{proposition}
  
  \begin{proof} This can be proven using the Ceresa cycle \cite[Example 3.3]{FLV2}. Let us however give an easy alternative proof that works for $g\ge 4$; this alternative proof will be relevant to the main result of the present paper.
  
Let $C$ be a very general curve of genus $g$ at least 4 and assume, by contradiction, that $C$ admits an MCK decomposition. Then this MCK decomposition must be defined by the 0-cycle $z={1\over 2g-2} K_C$ \cite[Proposition 3.1]{FLV2}, and so we have $K_C\in \A^1_{(0)}(C)$. We now consider $C\times C$ with the product MCK decomposition. Since an MCK decomposition is necessarily self-dual \cite[Proposition 2.7]{FLV2}, we have $\Delta_C\in \A^1_{(0)}(C\times C)$ and hence
  \[  \Delta_C\cdot (p_1)^\ast(K_C)\ \ \in\ \A^2_{(0)}(C\times C)\ .\]
  It follows that also 
  \[  Z_C:= \Delta_C\cdot (p_1)^\ast(K_C)  -   {1\over 2g-2} K_C\times K_C\ \   \in\ \A^2_{(0)}(C\times C)\ .\]
 The zero-cycle $Z_C$ is of degree 0 and $\A^2_{(0)}(C\times C)=\QQ[z\times z]$ is one-dimensional, and so we have
   \[  Z_C=0\ \ \ \hbox{in}\ \A^2(C\times C)\ .\]
   But this is absurd: the zero-cycle $Z_C$ (which is called the Faber--Pandharipande cycle, and which is the ``interesting 0-cycle'' in the title of \cite{GG}) is known to be non-zero in $\A^2(C\times C)$ for the very general curve of genus $g\ge 4$ (this is the main result of \cite{GG}, an alternative proof is given in \cite{Y2}). 
    \end{proof}    
    
    \begin{remark} The converse of Proposition \ref{hyper} is not true. That is, there exist non-hyperelliptic curves that have an MCK decomposition \cite{moi}, \cite{QZ}, \cite{LS}, \cite{LS2}.
     \end{remark} 
 
 \begin{remark}\label{alg} The argument using the Ceresa cycle given in \cite[Example 3.3]{FLV2} actually proves something stronger than Proposition \ref{no}: it proves that the very general curve of genus $\ge 3$ does not admit an MCK decomposition modulo algebraic equivalence.
 \end{remark}

 \subsection{The Franchetta property}

  \begin{definition}\label{def} Let $\XX\to B$ be a smooth projective morphism, where $\XX, B$ are smooth quasi-projective varieties, and let us write $X_b$ for the fiber over $b\in B$. We say that $\XX\to B$ has the {\em Franchetta property in codimension $j$\/} if the following holds: for every $\Gamma\in A^j(\XX)$ such that the restriction $\Gamma\vert_{X_b}$ is homologically trivial for the very general $b\in B$, the restriction $\Gamma\vert_{X_b}$ is zero in $A^j(X_b)$ for all $b\in B$.
 
 We say that $\XX\to B$ has the {\em Franchetta property\/} if $\XX\to B$ has the Franchetta property in codimension $j$ for all $j$.
 \end{definition}
 
 This property is studied in \cite{PSY}, \cite{BL}, \cite{FLV}, \cite{FLV3}.
 
 \begin{definition} Given a family $\XX\to B$ as in Definition \ref{def}, we will use the shorthand
   \[ \GDA^j_B(X_b):=\ima\bigl( \A^j(\XX)\to \A^j(X_b)\bigr)\ \ \ \subset\ \A^j(X_b) \ \]
   ($\GDA^\ast()$ stands for ``generically defined cycles'').
  \end{definition}
 
 With this definition, the Franchetta property for $\XX\to B$ is saying that generically defined cycles inject into cohomology.
 
\subsection{Franchetta for $Y$}

 \begin{notation}\label{not} Let $\Sigma^+$ be the spinor tenfold, and let $\OO_{\Sigma^+}(1)$ be the primitive polarization corresponding to the embedding
 $\Sigma^+\subset\PP(S^+)$. Let 
   \[ B\ \subset\ \bar{B}:=\PP H^0(\Sigma^+,\OO_{\Sigma^+}(1)^{\oplus 7})\] 
   denote the Zariski open subset parametrizing smooth dimensionally transverse linear sections of codimension 7.  
   
   Let
   \[ \YY\ \to\ B \]
   denote the universal family of smooth $3$-dimensional linear sections (in view of Theorem \ref{muk}, this is exactly the universal family of prime Fano threefolds of genus $7$).
  \end{notation}
 
 \begin{proposition}\label{Fr1} Let $\YY\to B$ be the universal family of prime Fano threefolds of genus $7$ (Notation \ref{not}). 
 The family $\YY\to B$ has the Franchetta property.
  \end{proposition}
 
 \begin{proof} The argument is inspired by \cite{PSY}, which is about the Franchetta property for K3 surfaces with a Mukai model.
 
 Let $\bar{\YY}\subset\bar{B}\times \Sigma^+$ denote the projective closure. As the line bundle $\OO_{\Sigma^+}(1)$ is base point free, the projection $\bar{\YY}\to \Sigma^+$ is a $\PP^r$-fibration.
 Let $Y=Y_b, b\in B$ be a smooth fiber.
 Using the projective bundle formula, one readily obtains (cf. for instance \cite[Proof of Lemma 1.1]{PSY}) that
   \[ \GDA^\ast_B(Y) =\ima\bigl( \A^\ast(\Sigma^+)\to \A^\ast(Y)\bigr)\ .\]
   Since $\A^j_{hom}(Y)=0$ for $j\not=2$, it only remains to ascertain that the cycle class map induces injections
     \begin{equation}\label{in} \ima\bigl(  \A^2(\Sigma^+)\to \A^2(Y)\bigr)\ \to\ H^4(Y,\QQ)=\QQ\ .\end{equation}
 It is known that $\A^2(\Sigma^+)$ is one-dimensional (see \cite{Tam}, or alternatively \cite[Theorem 4.16]{Ku3} where it is even proven that $\A^2(\Sigma^+)_{\ZZ}=\ZZ$), which ends the proof.
 \end{proof}
   
For later use, we observe the following: 

\begin{lemma}\label{gendef} The correspondences inducing the isomorphism of Proposition \ref{dualchow} are generically defined (with respect to $B$). In particular, there is an isomorphism
  \[ \GDA^2_B(Y)\  \cong\ \GDA^1_B(C)\oplus \QQ^2\ ,\]
  compatible with cycle class maps.
  \end{lemma}
  
  \begin{proof}
We use the fact (Theorem \ref{dual}(2)) that $Y$ and $C$ are related via homological projective duality (HPD). It is a general  fact that the Fourier--Mukai kernel defining the equivalence of Theorem \ref{dual}(2) exists relatively over $B$; actually this is inherent in the HPD package \cite{Kuz}. More precisely, the equivalence of Theorem \ref{dual}(2) is a consequence of the fact that $V:=\Sigma^+$ and $W:=\Sigma^-$ are HPD dual \cite{Ku0}. Let $Q(V,W)\subset V\times W$ denote the incidence correspondence, and let $P\in \hbox{Perf}(Q(V,W))$ be the Fourier--Mukai kernel defining the HPD functor $\Phi_P\colon D^b(W)\to D^b({\mathcal H}_{V})$ (in the language of \cite{Kuz}). Let $B\subset\bar{B}$ be the parameter space of smooth linear sections of codimension 7 as before, and let
  \[ Q_B(V,W):= Q(V,W)\times_{\bar{B}} B \]
  be the base change. The restriction 
  \[ P_B:= P\vert_{Q_B(V,W)}  \]
  defines a relative Fourier--Mukai functor over $B$. By the base change compatiblity of Fourier--Mukai functors \cite[Lemma 2.37 and Theorem 2.39]{Kuz}, for every $b\in B$, the restriction $P_b:=P_B\vert_b$ is the Fourier--Mukai kernel inducing the equivalence between the primitive components of $D^b(V_b)$ and $D^b(W_b)$ (in our setting,  $V_b$ is the Fano threefold $Y_b$, and $W_b$ is the curve $C_b$ dual to $Y_b$). Taking the relative Chern character, normalizing with the relative Todd class, and 
  pushing forward under the inclusion $Q_B(V,W)\subset \VV\times_B \WW$, one obtains a relative correspondence (over $B$) whose fiberwise restriction induces the isomorphism of motives $h^3(Y_b)\cong h^1(C_b)$. Since HPD duality is a duality relation, the inverse equivalence is obtained by exchanging $V$ and $W$, and so the inverse isomorphism of Chow motives is induced by the transpose correspondence.
 \end{proof}

%
%

 \subsection{Franchetta for $Y^2$?} The next step is to ask whether the family $\YY\times_B \YY\to B$ has the Franchetta property.

 \begin{proposition}\label{Fr2} Let $\YY\to B$ be as in Notation \ref{not}. Then
     \[ \GDA^\ast_B(Y^2) =  \bigl\langle  (p_1)^\ast(K_Y), (p_2)^\ast(K_Y), \Delta_Y\bigr\rangle\ .\]
     In particular, the family $\YY\times_B \YY\to B$ has the Franchetta property in codimension $\le 3$.
      \end{proposition} 
  
  \begin{proof} 
The line bundle $\OO_{\Sigma^+}(1)$ is very ample, and so this set-up verifies the property ($\ast_2$) of \cite{FLV}. This means that $\bar{\YY}\times_{\bar{B}} \bar{\YY}\to \Sigma^+\times\Sigma^+$ is a ``stratified projective bundle'' (in the sense of \cite{FLV}), and hence \cite[Proposition 5.2]{FLV} implies that there is equality
    \begin{equation}\label{gda} \GDA^\ast_B(Y\times Y) = \bigl\langle  \ima\bigl(\A^\ast(\Sigma^+\times \Sigma^+)\to \A^\ast(Y\times Y)\bigr),\, \Delta_Y\bigr\rangle\ .\end{equation}
   
As noted before, the spinor tenfold has trivial Chow groups (that is, the motive $h(\Sigma^+)$ is a sum of twisted Lefschetz motives, see \cite{Tam} or \cite[Theorem 4.16]{Ku3}) and so
  \[  \A^\ast   (\Sigma^+\times \Sigma^+)= \A^\ast (\Sigma^+)\otimes \A^\ast(\Sigma^+)\ .\]
 It follows that \eqref{gda} simplifies to
   \[ \begin{split}   \GDA^\ast_B(Y\times Y) &= \bigl\langle  \ima\bigl(\A^\ast(\Sigma^+\times \Sigma^+)\to \A^\ast(Y\times Y)\bigr),\, \Delta_Y\bigr\rangle\ \\
                                                                    &=  \bigl\langle (p_i)^\ast \ima\bigl( \A^\ast(\Sigma^+)\to \A^\ast(Y)\bigr),\, \Delta_Y\bigr\rangle   \\
                                                                    &=  \bigl\langle  (p_i)^\ast(K_Y),  \Delta_Y\bigr\rangle   \ , \\
                                                                    \end{split} \]
    where the last equality uses Proposition \ref{Fr1}.
    
    The Franchetta property in codimension $\le 2$ now follows immediately from the K\"unneth formula in cohomology. For codimension 3, the above 
    gives
      \[  \GDA^3_B(Y\times Y)=   \bigl\langle  (p_i)^\ast(K_Y)\bigr\rangle \oplus \QQ[\Delta_Y]\ .\]
      Since the cohomology class of the diagonal $\Delta_Y$ is linearly independent of the decomposable classes $ \bigl\langle  (p_i)^\ast(K_Y)\bigr\rangle $ (otherwise $H^3(Y,\QQ)$ would be zero, because decomposable correspondences act as zero on odd-degree cohomology), it follows that the Franchetta property holds in codimension 3.
         \end{proof}

  \begin{remark} As we will see below, the Franchetta property for $Y^2$ fails in codimension 4.
   \end{remark}

\section{Main result}

\begin{theorem}\label{main} Let $Y$ be a very general prime Fano threefold of genus 7, and let $H:=-K_Y\in \A^1(X)$. Then the cycle
   \[ \begin{split} Z_Y:=  \Delta_Y\cdot (p_1)^\ast(H) -   {1\over 12}&\Bigl( (p_1)^\ast(H )\cdot(p_2)^\ast(H^3)\\
       &+(p_1)^\ast(H^2)\cdot (p_2)^\ast(H^2 ) + (p_1)^\ast(H^3 )\cdot (p_2)^\ast(H )\Bigr)     \ \ \in\ \A^4(Y\times Y)\\
       \end{split} \]
   is (Abel--Jacobi trivial but) non-zero.
   \end{theorem}
   
 \begin{proof} It is readily checked (by letting the correspondence $Z_Y$ act on cohomology) that $Z_Y$ is homologically trivial.
 
 Next, let us assume by contradiction that $Z_Y$ is zero in $\A^4(Y\times Y)$. Then Proposition \ref{Fr2} would imply that 
   \[  \GDA^\ast_B(Y\times Y) =  \bigl\langle  (p_1)^\ast \GDA^\ast_B(Y),   (p_2)^\ast \GDA^\ast_B(Y)\bigr\rangle \oplus \QQ [\Delta_Y]\ .\]
   In view of Proposition \ref{Fr1} this would imply that $Y\times Y$ satisfies the Franchetta property.
   
   However, this is nonsensical, as can be seen on the side of the dual curve $C$. The motivic relation between $Y$ and $C$ (Proposition \ref{dualchow}), combined with Lemma \ref{gendef}, gives injective 
   maps
     \begin{equation}\label{inj} \GDA^i_B(C\times C) \ \hookrightarrow\ \GDA^{i+2}_B(Y\times Y) \end{equation}
     compatible with cycle class maps. Thus, $C\times C$ would also have the Franchetta property (with respect to $B$). But then the Faber--Pandharipande cycle
       \[   Z_C:= \Delta_C\cdot (p_1)^\ast(K_C)  -   {1\over 12} K_C\times K_C\ \   \in\ \A^2_{}(C\times C) \]
       would be zero, in contradiction with the result of Green--Griffiths \cite{GG}.
       
  The above argument actually shows that the map \eqref{inj} sends $Z_C$ to a non-zero multiple of $Z_Y$; as $Z_C$ is Albanese trivial this implies that $Z_Y$ is Abel--Jacobi trivial.     
  \end{proof}  
 
 \begin{corollary}\label{maincor} Let $Y$ be a very general prime Fano threefold of genus 7. Then $Y$ does not admit an MCK decomposition.
 \end{corollary}
 
 \begin{proof} 
 Since $Y$ is Kimura finite-dimensional, any CK decomposition for $Y$ must be of the form
   \[ \begin{split}  \pi^{2i}_Y &=  a_i \times b_i   \ \ \ \ (i=0,\ldots,3)\ ,\\
                            \pi^3_Y&= \Delta_Y - \sum_{i=0}^3 \pi^{2i}_Y\ ,\\
                            \end{split}\]
         for some cycles $a_i  \in \A_i(Y), b_i   \in \A^i(Y)$.

   Assuming there exists an MCK decomposition for $Y$, we have $H\in \A^1_{(0)}(Y)=\A^1(Y)$ and so by multiplicativity we must have
     \begin{equation}\label{0}    H^i \ \in\ \A^i_{(0)}(Y):= (\pi^{2i}_Y)_\ast \A^i(Y)\ .\end{equation}
     This forces the MCK decomposition to be of the form
     \[ \begin{split}  \pi^{2i}_Y &=  {1\over 12} H^{3-i} \times H^i\ ,\\
                            \pi^3_Y&= \Delta_Y - \sum_{i=0}^3 \pi^{2i}_Y\ .\\
                            \end{split}\]                        
   
   Assuming an MCK decomposition exists, we would have that
     \[  \Delta_Y\ \ \in\ \A^4_{(0)}(Y\times Y)\ ,\]
     with respect to the product MCK decomposition for $Y\times Y$ (this follows from the fact that MCK decompositions are self-dual, cf. \cite[Proof of Proposition 2.7]{FLV2}).    
    Combining with \eqref{0} and the definition of $Z_Y$, it would follow that
     \[ Z_Y\ \ \in\ \A^4_{(0)}(Y\times Y)\ .\]
     But 
       \[  \A^4_{(0)}(Y\times Y) =  \QQ[H^3\times H] \oplus \QQ[H^2\times H^2] \oplus \QQ[H\times H^3] \]
       is 3-dimensional, and injects into cohomology under the cycle class map. This would imply
       \[ Z_Y=0\ \ \hbox{in}\ \A^4(Y\times Y)\ ,\]
       in contradiction with Theorem \ref{main}.
      \end{proof}

\section{Cycles modulo algebraic equivalence}

In this section we consider the groups of algebraic cycles modulo algebraic equivalence 
  \[ \B^\ast(Y):= \A^\ast(Y)/{\sim_{\rm alg}}\ .\]

\subsection{An easy general result}

\begin{proposition}\label{main2} Let $Y$ be a smooth projective threefold with $\A_0(Y)=\QQ$ (in particular, $Y$ can be any Fano threefold). Then $Y$ admits an MCK decomposition modulo algebraic equivalence.
\end{proposition}

\begin{proof} The famous ``decomposition of the diagonal'' argument of Bloch--Srinivas implies that there exists a curve $C$ and a split injective map of motives
  \[  h(Y)\ \hookrightarrow\ h(C)(-1) \oplus \bigoplus \one(\ast)\ \ \hbox{in}\ \MM_{\rm rat}\ \]
  (see \cite{BS} or \cite[Theorem 3.11]{V3}). In particular, $Y$ is Kimura-finite and so has a CK decomposition $\{ \pi^i_Y\}$.
  
  Taking tensor products, we obtain a split injective map of motives
    \[ h(Y^3)\ \hookrightarrow\ h(C^3)(-3) \oplus \bigoplus h(C^2)(\ast) \oplus \bigoplus h(C)(\ast) \ \oplus \bigoplus \one(\ast)\ \ \hbox{in}\ \MM_{\rm rat}\ .\]
  In particular, it follows there is an injective map of groups of cycles
    \[  \B^6(Y^3)\ \hookrightarrow\ \B^3(C^3) \oplus \bigoplus \B^\ast(C^2) \oplus \bigoplus \B^\ast(C) \oplus \QQ^r  \ ,\]
   and this injection is compatible with cycle class maps.
    As algebraic and homological equivalence coincide for zero-cycles and for divisors, this implies the vanishing of the Griffiths group
    \[ \B^6_{hom}(Y^3)=0 \]
    (in other words, $\B^6(Y^3)$ injects into cohomology, under the cycle class map).
    This ensures the required vanishing
     \[ \pi_Y^k\circ \Delta_Y^{sm}\circ (\pi_Y^i\times \pi_Y^j)=0\ \ \ \hbox{in}\ \B^{6}(Y^3)\ \ \ \hbox{for\ all\ }i+j\not=k\ .\]
     \end{proof}

\begin{remark} The above argument applies more generally to any smooth projective variety $Y$ of dimension $2m+1\ge 3$ with the property that
  \[ \A_0^{hom}(Y)= \A_1^{hom}(Y)=\ldots= \A_m^{hom}(Y)=0\ .\]
Indeed, the Bloch-Srinivas argument (in the precise form given in \cite[Theorem 3.11]{V3}) provides a split injective map of motives
  \[   h(Y)\ \hookrightarrow\ h(C)(-m) \oplus \bigoplus \one(\ast)\ \ \hbox{in}\ \MM_{\rm rat}\ \]
  for some curve $C$. It follows that
    \[ \B^{4m+2}_{hom}(Y^3)\ \hookrightarrow\ \B^{m+2}_{hom}(C^3)  \oplus \bigoplus \B_{hom}^\ast(C^2) \oplus \bigoplus \B_{hom}^\ast(C)  =0\ ,\]
    and so $Y$ has an MCK decomposition modulo algebraic equivalence.
  \end{remark}

\subsection{Tautological cycles}

\begin{definition} Let $X$ be any smooth projective variety, and $m\in\NN$ a positive integer. The $\QQ$-subalgebra of tautological cycles is defined as
  \[  R^\ast(X^m):=\bigl\langle  (p_i)^\ast(K_X), (p_{ij})^\ast(\Delta_X)\bigr\rangle\ \ \subset\ \B^\ast(X^m)\ .\]
\end{definition}

Tautological cycles (inside the Chow ring $\A^\ast(X^m)$) have been studied for curves in \cite{Ta}, \cite{Ta2} and for cubic hypersurfaces in \cite{FLV3}.

We recall the following result:

\begin{proposition}\label{equiv} Let $Y$ be a Fano threefold with Picard number 1. The following are equivalent:

\begin{enumerate}

\item the subalgebra $R^\ast(Y^m)$ injects into cohomology, under the cycle class map, for all $m\in\NN$;

\item $Y$ has an MCK decomposition modulo algebraic equivalence.
\end{enumerate}
\end{proposition}

\begin{proof} This is proven modulo rational equivalence in \cite[Theorem 5.1]{more}; the same proof also works modulo algebraic equivalence.
\end{proof}

\begin{corollary}\label{taut} Let $Y$ be a very general prime Fano threefold of genus 7, and let $C$ be its dual curve. The injective map
  \[   \B^\ast( C^m)\ \hookrightarrow\ \B^\ast(Y^m) \]
  provided by Proposition \ref{dualchow} does not send $R^\ast(C^3)$ to $R^\ast(Y^3)$.
\end{corollary}

\begin{proof} This uses the fact that $C$ is a very general curve of genus 7 (Proposition \ref{surj}), and so $C$ does not admit an MCK decomposition modulo algebraic equivalence (see Remark \ref{alg}). Equivalently, the Gross--Schoen modified diagonal (as defined in the proof of Proposition \ref{hyper}) does not vanish modulo algebraic equivalence.
The tautological algebra $R^\ast(C^3)$ contains the Gross--Schoen modified diagonal 
  \[ \Gamma_3(C,{1\over 12}K_C)\ \ \in \B^2(C^3)\ ,\]
which is homologically trivial. By what we have just said, $ \Gamma_3(C^3,{1\over 12}K_C)$ is non-zero in $\B^2(C^3)$.

On the other hand, the tautological algebra $R^\ast(Y^m)$ injects into cohomology for all $m\in\NN$ (this follows from Proposition \ref{equiv} and Proposition \ref{main2}).

These two facts settle the corollary. Indeed, assume by contradiction that $R^\ast(C^3)$ is sent to $R^\ast(Y^3)$. Then $R^\ast(C^3)$ would inject into cohomology, and so the modified diagonal of $C$ would be zero in $\B^2(C^3)$, which is false.
\end{proof}

 \vskip1cm
\begin{nonumberingt} I am very grateful to the referee for helpful suggestions and corrections.
Many thanks to Kai for his beautiful interpretation of Schumann op. 18.
\end{nonumberingt}

\end{document}